\documentclass[12pt]{elsarticle}

\usepackage{latexsym,enumerate}
\usepackage{amsmath,amsthm,amsopn,amstext,amscd,amsfonts,amssymb}
\usepackage[ansinew]{inputenc}
\usepackage{verbatim}
\usepackage{graphicx}
\usepackage{dsfont,amsfonts}
\usepackage{epsfig} 
\usepackage{graphicx,psfrag} 
\usepackage{subfigure}
\usepackage{pstricks,pst-all,pst-node}
\usepackage{epstopdf}
\usepackage{color}
\usepackage{colortbl}
\usepackage{mathdots}
\usepackage{tikz}
\usetikzlibrary{arrows, automata}
\usetikzlibrary{decorations}
\usetikzlibrary{decorations.pathmorphing}
\usetikzlibrary{decorations.pathreplacing}
\usepackage{float}
\usetikzlibrary{positioning,chains,fit,shapes,calc}
\definecolor{myblue}{RGB}{0,0,0}
\definecolor{mygreen}{RGB}{0,0,0}

\newtheorem{theorem}{Theorem}[section]
\newtheorem{lemma}[theorem]{Lemma}
\newtheorem{proposition}[theorem]{Proposition}
\newtheorem{corollary}[theorem]{Corollary}

\newtheorem{algorithm}[theorem]{Algorithm}

\journal{}

\begin{document}

\begin{frontmatter}


	%

	\title{Encoding and Enumerating Acyclic Orientations of Graphs}


	\address[a]{Department of Mathematics and Statistics, Florida International University, 11200 SW 8th Street, Miami, FL 33199 U.S.A.}

	\author[a]{Walter Carballosa \corref{x}}
	\ead{wcarball@fiu.edu} \cortext[x]{Corresponding author.}

	\author[a]{Jessica Khera}
	\ead{jkhera@fiu.edu}

	\author[a]{Francisco Reyes}
	\ead{freyes@fiu.edu}

	\begin{abstract}
	In this work we study the acyclic orientations of graphs. We obtain an encoding of the acyclic orientations of the complete $p$-partite graph with size of its parts $n_1,n_2,\ldots,n_p$ via a vector with $p$ symbols and length $n=n_1+n_2+\ldots+n_p$ when the parts are fixed but not the vertices in each part. 
	We also give a recursive way to construct all acyclic orientations of a complete multipartite graph, this construction can be done by computer easily in order $\mathcal{O}(n)$. 
	Furthermore, we obtain a closed formula for non-isomorphic acyclic orientations of both the complete multipartite graphs and the complete multipartite graphs with a directed spanning tree. 
    Moreover, we obtain a closed formula for the number of acyclic orientations of a complete multipartite graph $K_{n_1,\ldots,n_p}$ with labelled vertices.  
    Finally, we obtain a way encode all acyclic orientations of an arbitrary graph as a permutation code. Using the codification mentioned above we obtain sharp upper and lower bounds of the number of acyclic orientations of a graph.
	\end{abstract}

\end{frontmatter}


\
\section{Introduction}
\label{sec:intro}

Throughout this paper, we will focus our attention on finite simple graphs $G$, with $V(G)$ denoting the set of vertices of the graph and $E(G)$ denoting the set of edges. A multipartite graph is a graph whose vertices can be partitioned into independent sets (parts). A multipartite graph with $k$ parts is referred to as a $k$-partite graph, and the complete $k$-partite graph is the graph with $k$ parts which contains an edge between every pair of vertices from different parts. This is denoted $K_{n_1, n_2, ... , n_k}$, where each $n_i$ is the number of vertices within part $V_i$.

An orientation of a graph is an assignment of a direction to each edge in some way. An orientation of a graph is said to be \emph{acyclic} if the assignment does not form any directed cycles. In an orientation of a graph a vertex is said to be a \emph{source} if its in-degree is zero, and similarly, a vertex is said to be a \emph{sink} if its out-degree is zero. We denote the multinomial coefficients by ${n \choose a_1, a_2, \dots, a_k}= \frac{n!}{a_1! \cdot a_2! \dots a_k!}$ where $n=\sum_{i=1}^k a_i$ and $a_1, a_2, \dots, a_k$ are nonnegative integers. 
A set of vertices of a graph $G$ is independent if its vertices are pairwise nonadjacent. The independence number of a graph $G$, denoted by $\alpha(G)$, is the cardinality of a largest independent set of $G$. 
Similarly, we denote the chromatic number of $G$, by $\chi(G)$. The chromatic number of $G$ is the smallest number of colors needed to color the vertices of $G$ so that no two adjacent vertices share the same color. 
Similarly, $\chi_G$ denotes the chromatic polynomial of $G$. The chromatic polynomial $\chi _{G}$ at a positive integer $\lambda$, $\chi_G(\lambda)$, is defined by the number of $\lambda$-colorings of the graph $G$. 

Acyclic orientations of a graph have been widely studied since the last century. One of the well-known results about acyclic orientations of a graph, the Gallai-Hasse-Roy-Vitaver theorem \cite[Theorem 3.13, p. 42]{NO}, was discovered (several times) during the 60's by these four authors above. The theorem states that the chromatic number of a connected graph equals one more than the length of the longest path of an acyclic orientation chosen to minimize this path length. 
In 1973, Stanley \cite{St} obtains that the number of acyclic orientations of a graph $G$ with order $n$ is $(-1)^n \chi_G(-1)$. 
In 1986, Linial \cite{L} used this result to prove that the number of acyclic orientation of a graph is an \#P-complete problem. 
In 1984, Johnson \cite{Jo} investigated the relationship between certain acyclic orientations and network reliability. In that paper, Johnson describes how generating the acyclic orientations of a network with a unique source can be used in computing its reliability. Other constructive methods for listing the acyclic orientation of a graph $G$ were investigated in \cite{PR,SSW,Sq1,Sq2}. Squire's algorithm in \cite{Sq2} requires $\mathcal{O}(n)$ time per acyclic orientation generated. 
Acyclic orientations of complete bipartite graphs were studied in \cite{KLM} and \cite{W}, and here, in this work, we extend to complete multipartite graphs.

In this paper we study the acyclic orientations of complete multipartite graphs. 
In Section \ref{Sect_complete}, we focus primarily on graphs with unlabelled vertices. We obtain a recursive way to construct all acyclic orientations of a complete multipartite graph with unlabelled vertices, and this construction can be done easily with a computer in order $\mathcal{O}(n)$ supported on the encoding given by Theorem \ref{th:coding} that is shown in \eqref{eq:code_function}. 
Additionally, Theorem \ref{th:non-isomorphics} counts the number of non-isomorphic acyclic orientations of the complete multipartite graphs; in particular, we obtain that $B(n_1,\ldots,n_p)=\frac{{n_1+\ldots+n_p\choose n_1,\ldots,n_p}}{r_1!r_2!\ldots r_s!}$ where $r_1,r_2,\ldots,r_s$ are the distinct numbers which represent the sizes of the parts in $K_{n_1,\ldots,n_p}$. Theorem \ref{th:span_trees} gives a closed formula for non-isomorphic acyclic orientations of the complete multipartite graphs which contain a directed spanning tree, $\mathcal C(n_1,\ldots,n_p)=\displaystyle \frac{{n_1+\ldots+n_p \choose n_1,\ldots,n_p}}{r_1!\ldots r_s!}  \cdot \frac{(n_1+\ldots+n_p)^2-(n_1^2+\ldots+n_p^2)}{(n_1+\ldots+n_p)^2-(n_1+\ldots+n_p)}$.
In Section \ref{Sect_labelled}, we deal with the acyclic orientations of complete multipartite graphs with labelled vertices. 
In this sense, we obtain a closed formula for the ordinary generating functions $\mathcal{F}(x_1,x_2,\ldots,x_p):=\displaystyle\sum_{k_1,\ldots,k_p\in\mathbb{N}} X_{k_1,k_2,\ldots,k_p}\,x_1^{k_1}x_2^{k_2}\ldots x_p^{k_p}$, where $X_{k_1,\ldots,k_p}$ is the number of strings in the alphabet $\{s_1,s_2,\ldots,s_p\}$ with $k_1$ characters $s_1$, $k_2$ characters $s_2$, and so on, with $k_p$ characters $s_p$ such that no two consecutive characters are the same. 
In Theorem \ref{th:acyc_labelled}, we obtain a closed formula for the number of acyclic orientation of a complete multipartite graph $K_{n_1,\ldots,n_p}$ with labelled vertices. 
Finally, we discuss the longest paths in acyclic orientations of complete multipartite graphs. Proposition \ref{p:length} gives that the length of the longest path in an acyclic orientation $\mathcal K$ of a complete multipartite graph is the size of the partition induced by $R_{\mathcal K}$ minus one. Furthermore, the number of longest paths in $\mathcal K$ is given by the multiplication of the sizes of each part of the partition induced by $R_{\mathcal K}$. The relation $R_{\mathcal K}$ is defined on $V(K_{n_1,\ldots,n_p})$ such that two vertices are related by $R_{\mathcal K}$ if the two vertices are represented in the code within a sub-string of consecutive and identical codes.

\section{Encoding acyclic orientations of complete multipartite graphs}\label{Sect_complete}

In this Section, we deal with encoding the acyclic orientations of the complete multipartite graphs with labelled parts but unlabelled vertices. The encoding we obtain allows us to count the number of non-isomorphic acyclic orientations of a complete multipartite graph and also for non-isomorphic acyclic orientations of a complete multipartite graphs which contain a directed spanning tree, \emph{i.e.}, an acyclic orientation with a unique source.

\begin{lemma}\label{l:ss}
 In every acyclic orientation of a complete multipartite graph there are both source and sink vertices.
 
 Furthermore, all sources (sinks, respectively) are vertices in the same part of the multipartite graph, whenever there are more than one source (sink, respectively).
\end{lemma}

This is a well-known result on acyclic orientations for an arbitrary graph. However, we present this particular proof for complete multipartite graphs to present self-contained work.

\begin{proof}
  Let us consider $K_{n_1,n_2,\ldots,n_p}$, a complete $p$-partite graph, $n_i\ge1$ for every $1\le i\le p$. We show here below a proof for the existence of a source. Note that the existence of a sink is analogous since by reversing the orientations in any acyclic orientation of a graph, the resulting orientation is still acyclic and a source in the reversed orientation is a sink in the original orientation. First, we claim the result is true for $p=2$.
  
  \medskip

  {\bf Claim 1:} \emph{There is a source in every acyclic orientation of $K_{n_1,n_2}$ for every $n_1,n_2\ge1$}.
  
  \noindent 
  We prove this by induction on $n_1+n_2\ge2$. Note that there is a source in an acyclic orientation of $K_{1,1}$. 
  Assume now that there is a source for every acyclic orientation of $K_{n_1,n_2}$ with $n_1+n_2=n\ge2$. 
  Consider an acyclic orientation of $G:=K_{n_1^*,n_2^*}$ with order $n_1^*+n_2^*=n+1$ and a vertex $v$ in Part 1 (Part with size $n_1^*$). 
  If one of the parts of $G$ has only one vertex, $G$ is a star; and consequently, we have either the central vertex of the Star is a source or one of the other vertices is. 
  Assume, $n_1^*,n_2^*>1$. 
  Hence, there is a source by assumption in the orientation obtained by removing $v$ from $G$, \emph{i.e.}, in $G^\prime:=G - \{v\}$. 
  If one of its sources is a vertex in Part 1, the result is done. 
  So, we can assume the source in $G^\prime$ is a vertex in Part 2. 
  Denote that vertex, $w$. 
  Note that if the orientation assigned to edge $vw$ is oriented from $w$ to $v$, then $w$ is a source in $G$, and the result is given. 
  Assume the orientation assigned to edge $vw$ is oriented from $v$ to $w$.
  There is a source by assumption in the orientation obtained by removing $w$ from $G$, \emph{i.e.}, in $G^{\prime\prime}:=G - \{w\}$. If one of the source in $G^{\prime\prime}$ is a vertex in Part 2, the result is done. We can assume the source in $G^{\prime\prime}$ is a vertex in Part 1. Denote that vertex, $v^\prime$.  
  If $v$ is a source of $G^{\prime\prime}$, then $v$ is a source in $G$. 
  Thus, assume $v$ is not a source in $G^{\prime\prime}$ and there exists a vertex $w^\prime$ in Part 2 such that edge $vw^\prime$ is oriented from $w^\prime$ to $v$ creating a cycle $v\to w\to v^\prime\to w^\prime$ which is impossible. Therefore, either $w$ or $v$ is a source in $G$. 
  Note that could be more than one source, but in that case all the sources are in the same part.

  \medskip

  {\bf Claim 2:} \emph{There is a source in every acyclic orientation of $K_{n_1,n_2,n_3}$ for every $n_1,n_2,n_3\ge1$}.
  
  \noindent 
  Let $A,B,C$ be the three parts of $K_{n_1,n_2,n_3}$ with corresponding sizes $n_1,n_2,n_3$, respectively. 
  Let $\mathcal{K}$ be an acyclic orientation of $K_{n_1,n_2,n_3}$.  
  Consider $A-B$ ($B-C$ and $C-A$, respectively) the acyclic orientation in $K_{n_1,n_2}$ ($K_{n_2,n_3}$ and $K_{n_3,n_1}$, respectively) considering only Parts $A\cup B$ ($B\cup C$ and $C\cup A$, respectively).
  By Claim 1, there is a source in each of $A-B$, $B-C$ and $C-A$. Since $\mathcal{K}$ is an acyclic orientation, the orientation of the sources cannot be all three $A\to B$, $B\to C$ and $C\to A$; nor its corresponding inverse orientations. 
  However, two of its sources (out of the three sub-orientations) must have consecutive orientations by the Pigeonhole Principle. Without loss of generality we can assume that there is a source in $A-B$ that is a vertex in Part $A$ and a source in $B-C$ that is a vertex in Part $B$, see Figure \ref{fig:2sources}.
  Thus, the source of $A-B$ in Part $A$ is also a source in $C-A$; otherwise, there is a cycle in $\mathcal{K}$ with length three.

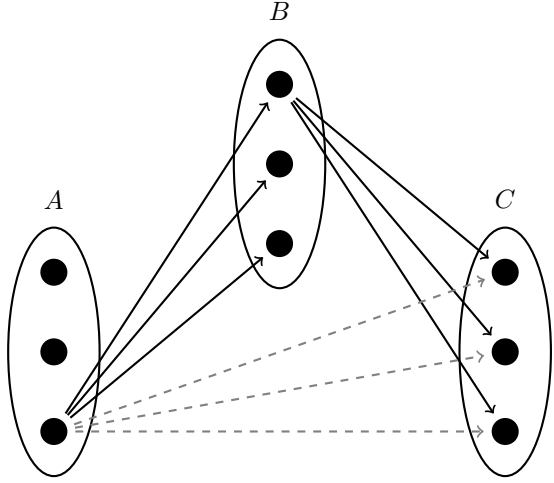
\begin{figure}
\centering
\begin{tikzpicture}[thick,
  every node/.style={draw,circle},
  fsnode/.style={fill=myblue},
  ssnode/.style={fill=mygreen},
  every fit/.style={ellipse,draw,inner sep=-2pt,text width=1cm},
  ->,shorten >= 3pt,shorten <= 3pt
]

\begin{scope}[start chain=going below,node distance=7mm]
\foreach \i in {1,2,3}
  \node[fsnode,on chain] (f\i){}; 
\end{scope}

\begin{scope}[xshift=3cm,yshift=2.5cm,start chain=going below,node distance=7mm]
\foreach \i in {4,5,6}
  \node[ssnode,on chain] (s\i){}; 
\end{scope}

\begin{scope}[xshift=6cm,yshift=0.0cm,start chain=going below,node distance=7mm]
\foreach \i in {7,8,9}
  \node[ssnode,on chain] (t\i){}; 
\end{scope}

\node [myblue,fit=(f1) (f3),label=above:$A$] {};
\node [mygreen,fit=(s4) (s6),label=above:$B$] {};
\node [mygreen,fit=(t7) (t9),label=above:$C$] {};

\draw (f3) -- (s4);
\draw (f3) -- (s5);
\draw (f3) -- (s6);
\draw (s4) -- (t7);
\draw (s4) -- (t8);
\draw (s4) -- (t9);
\draw[color=gray,dashed] (f3) -- (t7);
\draw[color=gray,dashed] (f3) -- (t8);
\draw[color=gray,dashed] (f3) -- (t9);
\end{tikzpicture}

\caption{Two sources in consecutive orientations $A\to B$ and $B\to C$ for visualizing the argument in Claim 2.}
\label{fig:2sources}
\end{figure}

  \medskip

  {\bf Claim 3:} \emph{There is a source in every acyclic orientation of $K_{n_1,n_2,\ldots,n_p}$ for every $n_1,n_2,\ldots,n_p\ge1$ and $p\ge2$}.
  
  \noindent 
  Claims 1 and 2 give the result for $p=2,3$. 
  Assume now that Claim 3 is true for some $p\ge3$.  
  Consider $G:=K_{n_1,\ldots,n_p,n_{p+1}}$ and let $\mathcal{K}$ be an acyclic orientation of $G$.
  Let $P_i$ be the part of $K_{n_1,\ldots,n_p,n_{p+1}}$ with corresponding sizes $n_i$, for every $1\le i\le p+1$. 
  Hence, there is a source by assumption in the orientation $\mathcal{K}^\prime$ obtained by removing the part $P_{p+1}$ from $G$, \emph{i.e.}, in $G^\prime:=G - P_{p+1}$. Without loss of generality we can assume that such sources are vertices in $P_1$. Similarly, there is a source by assumption in the orientation $\mathcal{K}^{\prime\prime}$ obtained by removing the part $P_{1}$ from $G$, \emph{i.e.}, in $G^{\prime\prime}:=G - P_{1}$. 
  So, without loss of generality we can assume that the sources in $G^{\prime\prime}$ are vertices in $P_2$. 
  Consider now the acyclic orientation $\mathcal{K}^{\prime\prime\prime}$ obtained from $G$ by removing all edges with both endpoints in $P_3\cup P_4\cup\ldots\cup P_{p+1}$. 
  Note that this is an acyclic orientation of a complete tripartite graph $K_{n_1,n_2,n_3+\ldots+n_p+n_{p+1}}$. 
  Thus, by Claim 2 there is a source in $\mathcal{K}^{\prime\prime\prime}$ that is a vertex in either $P_1$ or $P_2$. Therefore, this is a source of $\mathcal{K}$, too. 
  
  \medskip
  
  Finally, arguments above give that if there are more than one source (sink, respectively) then all sources (sinks, respectively) are vertices in the same part of the complete multipartite graph.
\end{proof}

Lemma \ref{l:ss} has the following direct consequences.

\begin{corollary}\label{c:ss}
 In every acyclic orientation of a non-empty graph there are both source and sink vertices.
\end{corollary}

\begin{proof}
 Let $\mathcal{K}$ be an acyclic orientation of a given graph $G(V,E)$. It is well-known that an acyclic orientation induces a partial ordering $<$ on the vertices of $G$, defined by two vertices $v_1, v_2\in V$ which satisfy $v_1 < v_2$ if and only if there is a directed path from $v_1$ to $v_2$ along the edges of the orientation. 
 Let's consider a listing $\mathcal{L}$ of the vertices of $G$ that preserve the partial ordering above induced by the acyclic orientation. Clearly, such a listing provides to $V$ with a total order. Hence, we can complete an acyclic orientation $\mathcal{K}^\prime$ of $K_n\simeq K_{\underbrace{1,\ldots,1}_{n}}$ from $\mathcal{K}$ by adding new directed edges according to ordering in $\mathcal{L}$, \emph{i.e.}, adding $(v_1,v_2)$ to $\mathcal{K}^\prime$ if $v_1$ precedes $v_2$ in $\mathcal{L}$. Thus, by Lemma \ref{l:ss} there are both source and sink vertices in $\mathcal{K}^\prime$ that are also a source and sink, respectively in $\mathcal{K}$, after removing the same directed edges.
\end{proof}

\begin{corollary}\label{c:tree}
 Let $G$ be an undirected graph without isolated vertices and let $\mathcal{G}$ be an acyclic orientation of $G$. Then, $\mathcal{G}$ contains a directed spanning tree if and only if $\mathcal{G}$ has a unique source.
\end{corollary}

\begin{proof}
 Note that if there is a spanning tree in $\mathcal{G}$ then the root of the spanning tree is the only source.
 Assume now that there is exactly one source in $\mathcal{G}$ and there does not exist a directed spanning tree in $\mathcal G$. 
 Denote $v$ the source in $\mathcal G$. 
 Consider the largest directed tree in $\mathcal G$ with root $v$, denoted $\mathcal T$. Since $\mathcal T$ is not a spanning tree, there is a non-empty collection of vertices in $G$ not included in $\mathcal T$, \emph{i.e.}, $V(G)\setminus V(\mathcal T)\neq\emptyset$. 
 Thus, since $G$ has no isolated vertices, every isolated vertex in $\mathcal R:=\langle V(G)\setminus V(\mathcal T)\rangle$ is a source in $\mathcal G$, but this is a contradiction. 
 If $\mathcal R$ has no isolated vertices, by Corollary \ref{c:ss}, $\mathcal R$ has a source $w\neq v$ that is also a source of $\mathcal G$. This is the contradiction we were looking for, therefore this completes the proof.
\end{proof}

For every collection of  natural numbers $n_1,n_2,\ldots,n_p$, with $p\ge2$, we define the number of acyclic orientations of a complete multipartite graph with fixed parts with respective sizes $n_1,n_2,\ldots,n_p$ and unlabelled vertices, by $\mathcal{A}(n_1,n_2,\ldots,n_p)$.
We define $\mathcal{A}(n_1,n_2,\ldots,n_p)$ such that one or more parts of the multipartite graph may be zero. This notation will be convenient. 
For instance, $\mathcal{A}(0,1,2)$ actually represents the number of acyclic orientations of a complete bipartite graph $K_{1,2}$ since the first part of the tripartite graph is empty. 
In the same way, $\mathcal{A}(1,0,2)$, $\mathcal{A}(0,1,0,0,2)$ and many others also represent $\mathcal{A}(1,2)$. 
Note that $\mathcal{A}(0,0,n)$ represents the number of acyclic orientations of a tripartite graph with at most one non-null (third) part with $n$ vertices, \emph{i.e.}, the empty graph $E_n$, and consequently, $\mathcal{A}(0,0,n)=1$ for all natural numbers $n$. 
For obvious reasons $\mathcal{A}(n_1,n_2,\ldots,n_p)$ is symmetric with respect to the permutation of their entries.

\begin{proposition}\label{p:recursion}
  $\mathcal{A}(n_1,n_2,\ldots,n_p)$ satisfies
  \begin{equation}\label{eq:recursion}
      \mathcal{A}(n_1,n_2,\ldots,n_p) = \mathcal{A}(n_1-1,n_2,\ldots,n_p)+\mathcal{A}(n_1,n_2-1,\ldots,n_p)+\ldots+\mathcal{A}(n_1,n_2,\ldots,n_p-1).
  \end{equation}
  with initial values $\mathcal{A}(n_1,n_2,\ldots,n_p)=0$ if $n_i=-1$ for some $1\le i\le p$, and $\mathcal{A}(0,\ldots,0)=1$.
\end{proposition}

\begin{proof}
 Consider a complete multipartite graph $K_{n_1,\ldots,n_p}$ and let $\mathcal{K}$ be one of its acyclic orientations. By Lemma \ref{l:ss} there is at least one source in $\mathcal{K}$ and all its sources are vertices in the same part. Let $v$ be one of the sources in $\mathcal{K}$. Without loss of generality we can assume that $v$ is a vertex in Part 1, $n_1\ge1$. 
 Note that the acyclic orientation $\mathcal{K}^\prime$ obtained by removing $v$ and the edges started at $v$ from $\mathcal{K}$ we obtain an acyclic orientation of $K_{n_1-1,n_2,\ldots,n_p}$. Since the arbitrary choice of the source may run over all parts, one at the time, and the parts of $K_{n_1,\ldots,n_p}$ are fixed, we obtain
 \[
 \mathcal{A}(n_1,n_2,\ldots,n_p) = \mathcal{A}(n_1-1,n_2,\ldots,n_p)+\mathcal{A}(n_1,n_2-1,\ldots,n_p)+\ldots+\mathcal{A}(n_1,n_2,\ldots,n_p-1).
 \]
 the initial values of the recursive formula \eqref{eq:recursion} follow from the natural meaning of $\mathcal{A}(0,\ldots,0)$, \emph{i.e.}, the number of acyclic orientations in a multipartite graph with empty parts of vertices $K_{0,\ldots,0}$. Clearly, from a part with no vertices there are no source to remove from, so  $\mathcal{A}(n_1,n_2,\ldots,n_p)=0$ if $n_i=-1$ for some $1\le i\le p$.
\end{proof}

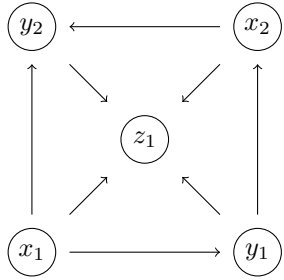
\begin{figure}[ht]
\centering   
\begin{tikzpicture}
\draw (0,0) circle (9pt);
\draw (3,3) circle (9pt);
\draw (3,0) circle (9pt);
\draw (0,3) circle (9pt);
\draw (1.5,1.5) circle (9pt);
\node at (0,0)  {$x_{1}$};
\node at (3,3)  {$x_{2}$};
\node at (3,0) {$y_{1}$};
\node at (0,3) {$y_{2}$};
\node at (1.5,1.5) {$z_{1}$};
\path[->] (0,.5) edge node {} (0,2.5);
\path[->] (.5,.5) edge node {} (1,1);
\path[->] (.5,0) edge node {} (2.5,0);
\path[->] (2.5,.5) edge node {} (2,1);
\path[->] (3,.5) edge node {} (3,2.5);
\path[->] (2.5,2.5) edge node {} (2,2);
\path[->] (.5,2.5) edge node {} (1,2);
\path[->] (2.5,3) edge node {} (.5,3);
\end{tikzpicture}

\caption{An acyclic orientation of $K_{2,2,1}$ (5-wheel graph) with consecutive sources: $x_1, y_1, x_2, y_2$ respectively} 
\label{fig:consecutive}
\end{figure}

Figure \ref{fig:consecutive} shows the sequence of sources in an acyclic orientation of the $5$-wheel graph, a complete tripartite graph $K_{2,2,1}$, providing a general view for an acyclic orientation of a multipartite graph. 
Proposition \ref{p:recursion} gives a direct procedure to generate all acyclic orientations of a multipartite graph, as well as encoding each orientation in an easy way. Define $[n]:=\{0,1,2,\ldots,n-1\}$ for every natural number $n$. 

\begin{theorem}\label{th:coding}
  There is a one-to-one correspondence from the set of all acyclic orientations of a complete multipartite graph with at most $p$ fixed parts and the $n^{th}$-vectors of $[p]^N$, where $N$ is the order of $G$. This correspondence assigns to each acyclic orientation with an $N$-vector code.
\end{theorem}

\begin{proof}
  Consider a complete multipartite graph $G=K_{n_0,n_1,\ldots,n_{p-1}}$ with $p\ge2$ parts labelled $0,1,2,\ldots,p-1$ where part $i$ has exactly $n_i$ unlabelled vertices for $i\in[p]$. 
  Let $\mathcal{L}$ be the set of all non-null acyclic orientations of complete multipartite graphs that are subgraphs of $G$, $\overline{\mathcal{L}}:=\mathcal{L}\cup \{K_0\}$. Note that,  $N=n_0+n_1+\ldots+n_{p-1}$. Conveniently consider that all vertices in an empty graph is both a source and a sink.
  Define the function $s:\mathcal{L}\to[p]$ as $s(\mathcal K)=i$ if the sources of $\mathcal K$ are vertices in the part $i$. 
  Define the function $f:\mathcal{L}\to\overline{\mathcal{L}}$ as $f(\mathcal K)$ is the acyclic orientation $\mathcal{K}^\prime$ obtained from $\mathcal K$ by removing one of the sources of $\mathcal K$. 
  Let's consider $\mathcal{K}$ a non-null acyclic orientation of $G$. 
  Clearly, the functions $s$ and $f$ are well defined.
  
  Now, we encode each acyclic orientation of $G$ as follows. 
  Thus, we recursively define the code function $C:\overline{\mathcal L}\to[p]^*$ for $\mathcal K$ as 
  \begin{equation}\label{eq:code_function}
   C(\mathcal K) = \left\{
    \begin{aligned}
     s(\mathcal K).C\left(f(\mathcal K)\right), & \text{ if } \mathcal{K} \text{ is non-null} \\
     \lambda, & \text{ if } \mathcal{K} \text{ is null} 
    \end{aligned}
   \right.
  \end{equation}
  where $\lambda$ is the empty string, $.$ denotes the concatenation of strings, and $[p]^*$ denotes the set of all strings in the alphabet $\{0,1,2,\ldots,p-1\}$. Note that the code function $C$ is well defined. Specifically, the code for an acyclic orientation of $G$ will be a string in $[p]^*$ with length $N$ since it concatenates one (source) vertex at the time until the null graph appears. 
\end{proof}

Similarly, we can define a code function like above using the sinks instead of the sources. Indeed, the code for an acyclic orientation $\mathcal K$ using the decomposition into a sequence of sinks is exactly the reverse code of $C(\mathcal K)$.

\begin{figure}[ht]
\setlength{\tabcolsep}{1mm} 
\def\arraystretch{1.0} 
\centering

\begin{tikzpicture}[thick,
  every node/.style={draw,circle},
  fsnode/.style={fill=myblue},
  ssnode/.style={fill=mygreen},
  every fit/.style={ellipse,draw,inner sep=-2pt,text width=2cm},
  ->,shorten >= 3pt,shorten <= 3pt
]
\begin{scope}[start chain=going below,node distance=7mm]
\foreach \i in {1,2}
  \node[fsnode,on chain] (f\i) [label=left: \i] {};
\end{scope}

\begin{scope}[xshift=4cm,yshift=0.5cm,start chain=going below,node distance=7mm]
\foreach \i in {3,4,5}
  \node[ssnode,on chain] (s\i) [label=right: \i] {};
\end{scope}

\node [myblue,fit=(f1) (f2),label=above:$n_0$] {};
\node [mygreen,fit=(s3) (s5),label=above:$n_1$] {};

\draw (s3) -- (f1);
\draw (f1) -- (s4);
\draw (f1) -- (s5);
\draw (f2) -- (s3);
\draw (f2) -- (s4);
\draw (f2) -- (s5);
\end{tikzpicture}

\caption{An acyclic orientation of $K_{2,3}$ with sequences of sources $(2,3,1,4,5)$ or $(2,3,1,5,4)$, and code $0.1.0.1.1$}
\label{fig:k23}
\end{figure}
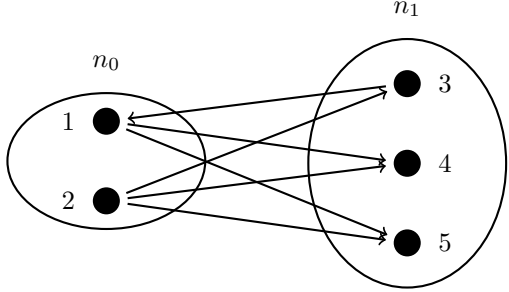

Figure \ref{fig:k23} shows the encoding of an acyclic orientation of $K_{2,3}$. Table \ref{tab:k22} shows all acyclic orientations of $K_{2,2}$ with its corresponding codes.
Encoding given in Theorem \ref{th:coding} allows us to obtain a closed formula for $\mathcal{A}(n_1,n_2,\ldots,n_p)$, see Theorem \ref{th:fixed_parts}.

\begin{table}[ht]
\setlength{\tabcolsep}{1mm} 
\def\arraystretch{1.0} 
\centering
\begin{tabular}{|c|c|c|c|}
\hline
\begin{tikzpicture}
 
\draw[fill=black] (0,0) circle (1.5pt);
\draw[fill=black] (1,0) circle (1.5pt);
\draw[fill=black] (0,1) circle (1.5pt);
\draw[fill=black] (1,1) circle (1.5pt);
\node at (-.25,.25)  {$x_{1}$};
\node at (1.25,.25)  {$y_{1}$};
\node at (-.25,1.25) {$x_{2}$};
\node at (1.25,1.25) {$y_{2}$};
\path[->] (0,0) edge node {} (.9,.9);
\path[->] (1,1) edge node {} (.1,1);
\path[->] (0,0) edge node {} (.9,0);
\path[->] (0,1) edge node {} (.9,.1);
\end{tikzpicture}
&\begin{tikzpicture}
\draw[fill=black] (0,0) circle (1.5pt);
\draw[fill=black] (1,0) circle (1.5pt);
\draw[fill=black] (0,1) circle (1.5pt);
\draw[fill=black] (1,1) circle (1.5pt);
\node at (-.25,.25)  {$x_{1}$};
\node at (1.25,.25)  {$y_{1}$};
\node at (-.25,1.25) {$x_{2}$};
\node at (1.25,1.25) {$y_{2}$};
\path[->] (0,0) edge node {} (.9,.9);
\path[->] (0,1) edge node {} (.9,1);
\path[->] (1,0) edge node {} (.1,0);
\path[->] (0,1) edge node {} (.9,.1);
\end{tikzpicture} 
&\begin{tikzpicture}
\draw[fill=black] (0,0) circle (1.5pt);
\draw[fill=black] (1,0) circle (1.5pt);
\draw[fill=black] (0,1) circle (1.5pt);
\draw[fill=black] (1,1) circle (1.5pt);
\node at (-.25,.25)  {$x_{1}$};
\node at (1.25,.25)  {$y_{1}$};
\node at (-.25,1.25) {$x_{2}$};
\node at (1.25,1.25) {$y_{2}$};
\path[->] (1,1) edge node {} (.1,.1);
\path[->] (0,1) edge node {} (.9,1);
\path[->] (0,0) edge node {} (.9,0);
\path[->] (0,1) edge node {} (.9,.1);
\end{tikzpicture} 
&\begin{tikzpicture}
\draw[fill=black] (0,0) circle (1.5pt);
\draw[fill=black] (1,0) circle (1.5pt);
\draw[fill=black] (0,1) circle (1.5pt);
\draw[fill=black] (1,1) circle (1.5pt);
\node at (-.25,.25)  {$x_{1}$};
\node at (1.25,.25)  {$y_{1}$};
\node at (-.25,1.25) {$x_{2}$};
\node at (1.25,1.25) {$y_{2}$};
\path[->] (0,0) edge node {} (.9,.9);
\path[->] (1,0) edge node {} (.1,.9);
\path[->] (0,0) edge node {} (.9,0);
\path[->] (0,1) edge node {} (.9,1);
\end{tikzpicture} 
\\ \( 0.1.0.1\) & \( 0.1.0.1\) & \( 0.1.0.1\) & \( 0.1.0.1\)
\\ \hline
\begin{tikzpicture}
\draw[fill=black] (0,0) circle (1.5pt);
\draw[fill=black] (1,0) circle (1.5pt);
\draw[fill=black] (0,1) circle (1.5pt);
\draw[fill=black] (1,1) circle (1.5pt);
\node at (-.25,.25)  {$x_{1}$};
\node at (1.25,.25)  {$y_{1}$};
\node at (-.25,1.25) {$x_{2}$};
\node at (1.25,1.25) {$y_{2}$};
\path[->] (0,1) edge node {} (.9,1);
\path[->] (1,1) edge node {} (.1,.1);
\path[->] (1,0) edge node {} (.1,.9);
\path[->] (1,0) edge node {} (.1,0);
\end{tikzpicture}
&\begin{tikzpicture}
\draw[fill=black] (0,0) circle (1.5pt);
\draw[fill=black] (1,0) circle (1.5pt);
\draw[fill=black] (0,1) circle (1.5pt);
\draw[fill=black] (1,1) circle (1.5pt);
\node at (-.25,.25)  {$x_{1}$};
\node at (1.25,.25)  {$y_{1}$};
\node at (-.25,1.25) {$x_{2}$};
\node at (1.25,1.25) {$y_{2}$};
\path[->] (1,1) edge node {} (.1,1);
\path[->] (1,1) edge node {} (.1,.1);
\path[->] (1,0) edge node {} (.1,.9);
\path[->] (0,0) edge node {} (.9,0);
\end{tikzpicture} 
&\begin{tikzpicture}
\draw[fill=black] (0,0) circle (1.5pt);
\draw[fill=black] (1,0) circle (1.5pt);
\draw[fill=black] (0,1) circle (1.5pt);
\draw[fill=black] (1,1) circle (1.5pt);
\node at (-.25,.25)  {$x_{1}$};
\node at (1.25,.25)  {$y_{1}$};
\node at (-.25,1.25) {$x_{2}$};
\node at (1.25,1.25) {$y_{2}$};
\path[->] (1,1) edge node {} (.1,1);
\path[->] (0,0) edge node {} (.9,.9);
\path[->] (1,0) edge node {} (.1,.9);
\path[->] (1,0) edge node {} (.1,0);
\end{tikzpicture} 
&\begin{tikzpicture}
\draw[fill=black] (0,0) circle (1.5pt);
\draw[fill=black] (1,0) circle (1.5pt);
\draw[fill=black] (0,1) circle (1.5pt);
\draw[fill=black] (1,1) circle (1.5pt);
\node at (-.25,.25)  {$x_{1}$};
\node at (1.25,.25)  {$y_{1}$};
\node at (-.25,1.25) {$x_{2}$};
\node at (1.25,1.25) {$y_{2}$};
\path[->] (1,1) edge node {} (.1,1);
\path[->] (1,1) edge node {} (.1,.1);
\path[->] (0,1) edge node {} (.9,.1);
\path[->] (1,0) edge node {} (.1,0);
\end{tikzpicture} 
\\ \( 1.0.1.0\) & \( 1.0.1.0\) & \( 1.0.1.0\) & \( 1.0.1.0\)
\\ \hline
\begin{tikzpicture}
\draw[fill=black] (0,0) circle (1.5pt);
\draw[fill=black] (1,0) circle (1.5pt);
\draw[fill=black] (0,1) circle (1.5pt);
\draw[fill=black] (1,1) circle (1.5pt);
\node at (-.25,.25)  {$x_{1}$};
\node at (1.25,.25)  {$y_{1}$};
\node at (-.25,1.25) {$x_{2}$};
\node at (1.25,1.25) {$y_{2}$};
\path[->] (0,1) edge node {} (.9,1);
\path[->] (1,1) edge node {} (.1,.1);
\path[->] (0,1) edge node {} (.9,.1);
\path[->] (1,0) edge node {} (.1,0);
\end{tikzpicture}
&\begin{tikzpicture}
\draw[fill=black] (0,0) circle (1.5pt);
\draw[fill=black] (1,0) circle (1.5pt);
\draw[fill=black] (0,1) circle (1.5pt);
\draw[fill=black] (1,1) circle (1.5pt);
\node at (-.25,.25)  {$x_{1}$};
\node at (1.25,.25)  {$y_{1}$};
\node at (-.25,1.25) {$x_{2}$};
\node at (1.25,1.25) {$y_{2}$};
\path[->] (1,1) edge node {} (.1,1);
\path[->] (0,0) edge node {} (.9,.9);
\path[->] (1,0) edge node {} (.1,.9);
\path[->] (0,0) edge node {} (.9,0);
\end{tikzpicture} 
&\begin{tikzpicture}
\draw[fill=black] (0,0) circle (1.5pt);
\draw[fill=black] (1,0) circle (1.5pt);
\draw[fill=black] (0,1) circle (1.5pt);
\draw[fill=black] (1,1) circle (1.5pt);
\node at (-.25,.25)  {$x_{1}$};
\node at (1.25,.25)  {$y_{1}$};
\node at (-.25,1.25) {$x_{2}$};
\node at (1.25,1.25) {$y_{2}$};
\path[->] (1,1) edge node {} (.1,1);
\path[->] (1,1) edge node {} (.1,.1);
\path[->] (0,1) edge node {} (.9,.1);
\path[->] (0,0) edge node {} (.9,0);
\end{tikzpicture} 
&\begin{tikzpicture}
\draw[fill=black] (0,0) circle (1.5pt);
\draw[fill=black] (1,0) circle (1.5pt);
\draw[fill=black] (0,1) circle (1.5pt);
\draw[fill=black] (1,1) circle (1.5pt);
\node at (-.25,.25)  {$x_{1}$};
\node at (1.25,.25)  {$y_{1}$};
\node at (-.25,1.25) {$x_{2}$};
\node at (1.25,1.25) {$y_{2}$};
\path[->] (0,1) edge node {} (.9,1);
\path[->] (0,0) edge node {} (.9,.9);
\path[->] (1,0) edge node {} (.1,.9);
\path[->] (1,0) edge node {} (.1,0);
\end{tikzpicture} 
\\ \( 0.1.1.0\) & \( 0.1.1.0\) & \( 1.0.0.1\) & \( 1.0.0.1\)
\\ \hline
\begin{tikzpicture}
\draw[fill=black] (0,0) circle (1.5pt);
\draw[fill=black] (1,0) circle (1.5pt);
\draw[fill=black] (0,1) circle (1.5pt);
\draw[fill=black] (1,1) circle (1.5pt);
\node at (-.25,.25)  {$x_{1}$};
\node at (1.25,.25)  {$y_{1}$};
\node at (-.25,1.25) {$x_{2}$};
\node at (1.25,1.25) {$y_{2}$};
\path[->] (0,0) edge node {} (.9,.9);
\path[->] (0,1) edge node {} (.9,1);
\path[->] (0,0) edge node {} (.9,0);
\path[->] (0,1) edge node {} (.9,.1);
\end{tikzpicture}
&\begin{tikzpicture}
\draw[fill=black] (0,0) circle (1.5pt);
\draw[fill=black] (1,0) circle (1.5pt);
\draw[fill=black] (0,1) circle (1.5pt);
\draw[fill=black] (1,1) circle (1.5pt);
\node at (-.25,.25)  {$x_{1}$};
\node at (1.25,.25)  {$y_{1}$};
\node at (-.25,1.25) {$x_{2}$};
\node at (1.25,1.25) {$y_{2}$};
\path[->] (1,1) edge node {} (.1,1);
\path[->] (1,1) edge node {} (.1,.1);
\path[->] (1,0) edge node {} (.1,.9);
\path[->] (1,0) edge node {} (.1,0);
\end{tikzpicture} 
&\begin{tikzpicture}
\draw[fill=black] (0,0) circle (1.5pt);
\draw[fill=black] (1,0) circle (1.5pt);
\draw[fill=black] (0,1) circle (1.5pt);
\draw[fill=black] (1,1) circle (1.5pt);
\node at (-.25,.25)  {$x_{1}$};
\node at (1.25,.25)  {$y_{1}$};
\node at (-.25,1.25) {$x_{2}$};
\node at (1.25,1.25) {$y_{2}$};
\path[->] (0,1) edge node {} (.9,1);
\path[->] (1,1) edge node {} (.1,.1);
\path[->] (1,0) edge node {} (.1,.9);
\path[->] (0,0) edge node {} (.9,0);
\end{tikzpicture} 
&\begin{tikzpicture}
\draw[fill=black] (0,0) circle (1.5pt);
\draw[fill=black] (1,0) circle (1.5pt);
\draw[fill=black] (0,1) circle (1.5pt);
\draw[fill=black] (1,1) circle (1.5pt);
\node at (-.25,.25)  {$x_{1}$};
\node at (1.25,.25)  {$y_{1}$};
\node at (-.25,1.25) {$x_{2}$};
\node at (1.25,1.25) {$y_{2}$};
\path[->] (0,0) edge node {} (.9,.9);
\path[->] (1,1) edge node {} (.1,1);
\path[->] (1,0) edge node {} (.1,0);
\path[->] (0,1) edge node {} (.9,.1);
\end{tikzpicture} 
\\ \( 0.0.1.1\) & \( 1.1.0.0\) & \( \lambda \) & \( \lambda \)
\\ \hline
\end{tabular}
\caption{All orientations of $K_{2,2}$ with labeled vertices and their corresponding codes}
\label{tab:k22}
\end{table}

\begin{theorem}\label{th:fixed_parts}
  \begin{equation}\label{eq:fixed_parts}
  \mathcal{A}(n_1,n_2,\ldots,n_p)={n_1+n_2+\ldots+n_p \choose n_1,n_2,\ldots,n_p}:=\frac{(n_1+n_2+\ldots+n_p)!}{n_1!n_2!\ldots n_p!}.
  \end{equation}
\end{theorem}

\begin{proof}
 The result is a direct consequence of counting the number of codes as in Theorem \ref{th:coding}, \emph{i.e.}, numbers in the numerical system of base $p$ with length $n_1+n_2+\ldots+n_p$ and with exactly $n_1$ digits $0$, $n_2$ digits $1$, \ldots, $n_p$ digits $p-1$.
\end{proof}

We can also obtain the result above by proving that \eqref{eq:fixed_parts} satisfies the recurrence given in Proposition \ref{p:recursion}.
Let us now denote by $\mathcal{B}(n_1,n_2,\ldots,n_p)$ the number of non-isomorphic acyclic orientations of a complete multipartite graph $K_{n_1,n_2,\ldots,n_p}$.
The encoding given in Theorem \ref{th:coding} allows us to obtain the number of non-isomorphic acyclic orientations of complete multipartite graphs, see Theorem \ref{th:non-isomorphics}.

\begin{table}[ht]
\setlength{\tabcolsep}{1mm} 
\def\arraystretch{1.0} 
\centering
\begin{tabular}{|c|c|c|}
\hline
\begin{tikzpicture}[thick,
  every node/.style={draw,circle},
  fsnode/.style={fill=myblue},
  ssnode/.style={fill=mygreen},
  every fit/.style={ellipse,draw,inner sep=2pt,text width=1cm},
  ->,shorten >= 4pt,shorten <= 4pt
]
\begin{scope}[start chain=going below,node distance=7mm]
\foreach \i in {1,2}
  \node[fsnode,on chain] (f\i) [] {};
\end{scope}

\begin{scope}[xshift=2.5cm,yshift=0cm,start chain=going below,node distance=7mm]
\foreach \i in {3,4}
  \node[ssnode,on chain] (s\i) [] {};
\end{scope}

\node [myblue,fit=(f1) (f2),label=above:$n_0$] {};
\node [mygreen,fit=(s3) (s4),label=above:$n_1$] {};

\draw (f1) -- (s3);
\draw (f1) -- (s4);
\draw (f2) -- (s3);
\draw (f2) -- (s4);
\end{tikzpicture}

&\begin{tikzpicture}[thick,
  every node/.style={draw,circle},
  fsnode/.style={fill=myblue},
  ssnode/.style={fill=mygreen},
  every fit/.style={ellipse,draw,inner sep=2pt,text width=1cm},
  ->,shorten >= 4pt,shorten <= 4pt
]
\begin{scope}[start chain=going below,node distance=7mm]
\foreach \i in {1,2}
  \node[fsnode,on chain] (f\i) [] {};
\end{scope}

\begin{scope}[xshift=2.5cm,yshift=0cm,start chain=going below,node distance=7mm]
\foreach \i in {3,4}
  \node[ssnode,on chain] (s\i) [] {};
\end{scope}

\node [myblue,fit=(f1) (f2),label=above:$n_0$] {};
\node [mygreen,fit=(s3) (s4),label=above:$n_1$] {};

\draw (s3) -- (f1);
\draw (s4) -- (f1);
\draw (f2) -- (s3);
\draw (f2) -- (s4);
\end{tikzpicture}

&\begin{tikzpicture}[thick,
  every node/.style={draw,circle},
  fsnode/.style={fill=myblue},
  ssnode/.style={fill=mygreen},
  every fit/.style={ellipse,draw,inner sep=2pt,text width=1cm},
  ->,shorten >= 4pt,shorten <= 4pt
]
\begin{scope}[start chain=going below,node distance=7mm]
\foreach \i in {1,2}
  \node[fsnode,on chain] (f\i) [] {};
\end{scope}

\begin{scope}[xshift=2.5cm,yshift=0cm,start chain=going below,node distance=7mm]
\foreach \i in {3,4}
  \node[ssnode,on chain] (s\i) [] {};
\end{scope}

\node [myblue,fit=(f1) (f2),label=above:$n_0$] {};
\node [mygreen,fit=(s3) (s4),label=above:$n_1$] {};

\draw (f1) -- (s3);
\draw (s4) -- (f1);
\draw (f2) -- (s3);
\draw (f2) -- (s4);
\end{tikzpicture}
\\ \( 0.0.1.1\)&\(0.1.1.0\)&\(0.1.0.1\)
\\ \hline
\end{tabular}

\caption{All non-isomorphic acyclic orientations of $K_{2,2}$ with unlabeled vertices and their corresponding codes}
\label{tab:isomorph}
\end{table}


\begin{theorem}\label{th:non-isomorphics}
  \[
  \mathcal{B}(n_1,n_2,\ldots,n_p)=\frac{{n_1+n_2+\ldots+n_p \choose n_1,n_2,\ldots,n_p}}{r_1! r_2!\ldots r_s!}
  \]
  where ${n_1+n_2+\ldots+n_p \choose n_1,n_2,\ldots,n_p}$ is the multinomial coefficient and $r_1,r_2,\ldots,r_s$ are the number of parts in $K_{n_1,n_2,\ldots,n_p}$ grouped by the same size (i.e. $r_i$ for $i \in \{1,...,s\}$ suggests there are $r_i$ many parts that contain the same fixed number of vertices).
\end{theorem}

\begin{proof}
 Let $\mathcal{K}_1,\mathcal{K}_2$ be two isomorphic acyclic orientations of a multipartite graph $G$. 
 Note that every isomorphism $\sigma:V(\mathcal K_1)\to V(\mathcal K_2)$ matches all vertices of each part of $\mathcal K_1$ with all vertices in a part of $\mathcal K_2$; otherwise, $\sigma$ doesn't preserve adjacency. Therefore, $\mathcal{A}(n_1,\ldots,n_k)=\mathcal{B}(n_1,\ldots,n_k)$ if $n_i\neq n_j$ when $i\neq j$ since the vertices in each part are unlabelled. 
 Then $\sigma$ matches all sources of $\mathcal{K}_1$ with the sources of $\mathcal{K}_2$. Moreover, if we remove all sources from both $\mathcal{K}_1$ and $\mathcal{K}_2$ obtaining $\mathcal{K}_1^\prime$ and $\mathcal{K}_2^\prime$, respectively, then $\sigma$ matches the sources of $\mathcal{K}_1^\prime$ and $\mathcal{K}_2^\prime$ as well. We can repeat this removing procedure until a null-graph is obtained.
 Hence, the codes assigned to $\mathcal{K}_1$ and $\mathcal{K}_2$ by Theorem \ref{th:coding} match each other except by a possible permutation within of the parts with the same size. 
 Besides, if there are exactly $r_1$ parts with the same size in $G$, we have $r_1!$ different codes associated to isomorphic acyclic orientations when the (digits) codes assigned to the $r_1$ equal-size parts are permuted and kept the same places for the other digits. Therefore, by grouping the parts of $G$ with the same size, we obtain the result.
\end{proof}

 From Theorem 2.7, we have that $\mathcal{B}(2,2) = 3$. Table \ref{tab:isomorph} shows all 3 non-isomorphic acyclic orientations of $K_{2,2}$ with their corresponding codes. Note that by swapping the digits $0$'s and $1$'s we obtain isomorphic acyclic orientations due to both parts have the same size. Permuting vertices within each part also obtains an isomorphic orientation. 
 Let us denote by $\mathcal{C}(n_1,n_2,\ldots,n_p)$ the number of non-isomorphic acyclic orientations of a complete multipartite graph $K_{n_1,n_2,\ldots,n_p}$ which contains a directed spanning tree.
 Similarly, we can obtain the number of non-isomorphic acyclic orientations of complete multipartite graphs containing a directed spanning tree, see Theorem \ref{th:span_trees}.

\begin{theorem}\label{th:span_trees}
  Let $n_1,\ldots,n_p$ be $p$ positive integer numbers and $N=n_1+\ldots+n_p$. We have
  \[
    \mathcal C(n_1,\ldots,n_p)=\displaystyle \frac{{N \choose n_1,\ldots,n_p}}{r_1!\ldots r_s!}  \cdot \frac{N^2-\sum_{r=1}^p (n_r)^2}{N(N-1)}.
  \]
  where $r_1,r_2,\ldots,r_s$ are the number of parts in $K_{n_1,n_2,\ldots,n_p}$ grouped by the same size.
\end{theorem}

\begin{proof}
 The result is a direct consequence of counting the number of codes as in Theorem \ref{th:coding}, \emph{i.e.}, numbers in the numerical  system of base $p$ with length $N:=n_1+n_2+\ldots+n_p$ and with exactly $n_1$ digits $0$, $n_2$ digits $1$, \ldots, $n_p$ digits $p-1$ such that the first two digits are distinct. That is
 \[
  \displaystyle\sum_{1\le i<j\le p} 2\,\mathcal{A}(\ldots,n_i-1,\ldots,n_j-1,\ldots)= \frac{\mathcal{A}(n_1,\ldots,n_p)}{N(N-1)} \, \sum_{1\le i<j\le p}\,2n_i n_j
 \]
 Then, by considering the isomorphic directed graphs obtained by permuting the parts with the same size, we obtain
 \[
  \mathcal C(n_1,\ldots,n_p)=\displaystyle \frac{\mathcal{A}(n_1,\ldots,n_p)}{r_1!\ldots r_s!} \, \frac{N^2-\sum_{r=1}^p (n_r)^2}{N(N-1)}.
 \]
\end{proof}

\
\section{Acyclic orientations of a complete multipartite graph with labelled vertices}\label{Sect_labelled}

In the previous section, we considered multipartite graphs with unlabeled vertices. In this section we deal with the number of acyclic orientations of complete multipartite graphs with labelled vertices. 
Note that using the coding in \eqref{eq:code_function} we can also obtain the poly-Bernoulli numbers $B_{n_1,n_2}$, \emph{i.e.}, number of acyclic orientation of a labelled complete bipartite graph $K_{n_1,n_2}$ with size of each part $n_1$ and $n_2$, respectively, see Proposition \ref{p:labelled}. It is well-known that poly-Bernoulli number also counts, for instance, the lonesum matrices \cite{Br} and Callan permutations \cite{BH} among other things.
The closed formula below was given by Arakawa and Kaneko in \cite{AK}.

\begin{proposition}\label{p:labelled}
 The number of acyclic orientations of a complete bipartite graph $K_{n_1,n_2}$ with labeled vertices and size of each part $n_1$ and $n_2$, respectively, is
 \begin{equation}\label{eq:labelled}
     B_{n_1,n_2}=\displaystyle\sum_{m=0}^{\min\{n_1,n_2\}} \, (m!)^2\,{n_1+1\brace m+1}\,{n_2+1\brace m+1},
 \end{equation}
where ${r\brace s}$ denotes the Stirling number of the second kind. 
\end{proposition}

\begin{proof}
 Without loss of generality we can assume that $n_1\leq n_2$.
 Notice that we may count the number of acyclic orientations of $K_{n_1,n_2}$ by grouping them according to $m\ge1$, the number of groups of $0$'s codes (code $0$ associated to part with size $n_1$) separated by at least a code $1$. In this case, the number of groups of $1$'s codes separated by at least a code $0$ must be either $m-1$, $m$ or $m+1$. 
 We recall ${r\brace s}$ is the number of partitions of a set with $r$ elements into $s$ parts is ${r\brace s}$, and the well-known identity
 \[
  {r+1\brace s}=s{r\brace s}+{r\brace s-1} \qquad \text{for all }r,s\ge0.
 \]
 Thus we count it by consider the number of partitions of the $n_1$ elements in Part 1 and the number of partitions of the $n_2$ elements in Part 2; follow by multiplying by the corresponding factorial given by the permutation of the distinct groups in the partition of each part. Then, we have
 \begin{align*}
     &\displaystyle\sum_{m=1}^{n_1} \left( m!(m-1)!{n_1\brace m}{n_2\brace m-1} + 2\,m!m!{n_1\brace m}{n_2\brace m} + m!(m+1)!{n_1\brace m}{n_2\brace m+1} \right) \\
     = & \sum_{m=1}^{n_1} m!(m-1)!{n_1\brace m}\left( {n_2\brace m-1} + m\,{n_2\brace m}\right) + \sum_{m=1}^{n_1} (m!)^2\,{n_1\brace m}\left( {n_2\brace m} + (m+1)\,{n_2\brace m+1}\right) \\
     =& \sum_{m=1}^{n_1} m!(m-1)!{n_1\brace m}{n_2+1\brace m} + \sum_{m=1}^{n_1} (m!)^2\,{n_1\brace m}{n_2+1\brace m+1} \\
     =& \sum_{m=0}^{n_1-1} (m+1)!m!{n_1\brace m+1}{n_2+1\brace m+1} + \sum_{m=1}^{n_1} (m!)^2\,{n_1\brace m}{n_2+1\brace m+1} \\
     =& \sum_{m=0}^{n_1} (m+1)!m!{n_1\brace m+1}{n_2+1\brace m+1} + \sum_{m=0}^{n_1} (m!)^2\,{n_1\brace m}{n_2+1\brace m+1} \\
     =& \sum_{m=0}^{n_1} (m!)^2\left( (m+1)\,{n_1\brace m+1} + {n_1\brace m} \right){n_2+1\brace m+1} \\
     =& \sum_{m=0}^{n_1} (m!)^2{n_1+1\brace m+1} {n_2+1\brace m+1} \\
     =& \,B_{n_1,n_2}
. \end{align*}
\end{proof}

In order to obtain the result for the number of acyclic orientation of a labelled multipartite graph with $p\geq3$ parts in a similar way, we define $X_{k_1,k_2,\ldots,k_p}$ by the number of strings in the alphabet $\mathcal{S}:=\{s_1,s_2,\ldots,s_p\}$ with $k_1$ characters $s_1$, $k_2$ characters $s_2$, and so on with $k_p$ characters $s_p$ such that no two consecutive characters are the same. We define  $X^{(i)}_{k_1,k_2,\ldots,k_p}$ by the number of strings in $\mathcal{S}$ with $k_1$ characters $s_1$, $k_2$ characters $s_2$, and so on with $k_p$ characters $s_p$ such that there are no two consecutive identical characters and the first character is $s_i$ for $1\leq i\leq p$. Clearly, we have 
\begin{equation}\label{eq:D-ABC}
    X_{k_1,\ldots,k_p}=X^{(1)}_{k_1,\ldots,k_p}+X^{(2)}_{k_1,\ldots,k_p}+\ldots+X^{(p)}_{k_1,\ldots,k_p}.
\end{equation}

Note that for some $p$-tuples $(k_1,\ldots,k_p)\in\mathbb{N}^p$, we have $X_{k_1,\ldots,k_p}=0$, for instance, $X_{2,0,\ldots,0}=0$ since we cannot alternate two characters $s_1$ and no other characters without leaving two consecutive characters $s_1$. Moreover, $X_{k_1,\ldots,k_p}>0$ if and only if $(k_1,\ldots,k_p)\in\mathbb{T}$ where 
\[
\mathbb{T}:=\left\{(k_1,\ldots,k_p)\in\mathbb{N}^p \,:\, \max\{k_1,k_2,\ldots,k_p\} \leq \frac{1 + \sum_{i=1}^{p}k_i}2 \right\}.
\]
Hence, we have

\begin{equation}\label{eq:disjoint}
        X^{(j)}_{k_1,k_2,\ldots,k_p} \, = \displaystyle\sum_{i=1}^{p} \, X^{(i)}_{k_1,\ldots,k_j-1,\ldots,k_p} \,  - \, X^{(j)}_{k_1,\ldots,k_j-1,\ldots,k_p} \, + \chi_{\{e_j\}} (k_1,k_2,\ldots,k_p)
\end{equation}
for all  $k_1,k_2,\ldots,k_p\in \mathbb{N}$ and $1\le j\le p$ where $e_j$ represents the $j^{th}$ canonical vector of $\mathbb{R}^p$ and $\chi_A$ is the indicator of $A$. 
For obvious reasons consider $X_{k_1,k_2,\ldots,k_p}=X^{(j)}_{k_1,k_2,\ldots,k_p}=0$ if $k_i<0$ for some $1\le i\le p$ and for every $1\le j\le p$. Note that, on the one hand, if $X^{(1)}_{k_1,k_2,\ldots,k_p}>0$, by removing the first character (a $s_1$) from each string counted in $X^{(1)}(k_1,k_2,\ldots,k_p)$ we obtain a string counted in $X^{(j)}_{k_1-1,k_2,\ldots,k_p}$ for some $2\le j\le p$, except when $(k_1,k_2,\ldots,k_p)=(1,0,\ldots,0)$ where we consider that the empty string $\lambda$ did not count. Besides, by adding a character $s_1$ to the front of each string counted in $X^{(j)}_{k_1-1,k_2,\ldots,k_p}$ for $2\le j\le p$ we obtain distinct strings in $X^{(1)}_{k_1,k_2,\ldots,k_p}$; analogously, we obtain the corresponding relations for $X^{(j)}_{k_1,k_2,\ldots,k_p}$ for $2\le j\le p$, respectively. 
On the other hand, if $X^{(1)}_{k_1,k_2,\ldots,k_p}=0$, then we have either $k_1>1+\sum_{i=2}^{p} k_i$ and consequently $X^{(j)}_{k_1-1,k_2,\ldots,k_p}=0$ for $2\le j\le p$, or there are more characters $s_j$ than the other characters for some $2\le j\le p$ making $X^{(i)}_{k_1-1,k_2,\ldots,k_p}=0$ for every $2\le i\le p$. We can obtain analogous equations to \eqref{eq:disjoint}. 

Now we define a $p$-variables ordinary generating function 
\[
  \mathcal{F}(x_1,\ldots,x_p):=\displaystyle\sum_{k_1,\ldots,k_p\in\mathbb{N}} X_{k_1,\ldots,k_p}\,x_1^{k_1}x_2^{k_2}\cdot\ldots\cdot x_p^{k_p}.
\]
Note that $\mathcal{F}(k_1,\ldots,k_p)$ converges absolutely on $|x_1|+\ldots+|x_p|<1$ since $0\leq X_{k_1,\ldots,k_p}\leq {k_1+\ldots+k_p \choose k_1,\ldots,k_p}$ for all $k_1,\ldots,k_p\in\mathbb{N}$.

\begin{proposition}\label{p:generating}
    We have
    \[
      \mathcal{F} (x_1,\ldots,x_p) = \displaystyle\frac{ \frac{x_1}{x_1+1} + \ldots + \frac{x_p}{x_p+1} }{ 1 - \left(\frac{x_1}{x_1+1} + \ldots + \frac{x_p}{x_p+1}\right) }.
    \]
\end{proposition}

\begin{proof}
    Since $\mathcal F (x_1,\ldots,x_p)$ converges absolutely in a domain including $\Omega:=\{(x_1,\ldots,x_p)\in\mathbb{R}^p \,:\, |x_1|+\ldots+|x_p|<1\}$. We may reorder its terms of summation without affecting the sum.
    Define now
    \[
        \mathcal{F}^{(j)}(x_1,\ldots,x_p)  :=\displaystyle\sum_{k_1,\ldots,k_p\in\mathbb{N}} X^{(j)}_{k_1,\ldots,k_p}\,x_1^{k_1}x_2^{k_2}\ldots x_p^{k_p} \quad \text{for every } 1\le j\le p.
    \]
    
    Now by performing the summation of \eqref{eq:disjoint} for every $(k_1,\ldots,k_p)\in\mathbb{N}^p$, we obtain
    \begin{equation}\label{eq:gen1}
        \mathcal{F}^{(j)}(x_1,\ldots,x_p)  =x_j\, \displaystyle\sum_{i=1}^{p} \mathcal{F}^{(i)}(x_1,\ldots,x_p) - x_j\, \mathcal{F}^{(j)}(x_1,\ldots,x_p) + x_j  \quad \text{for every } 1\le j\le p.
    \end{equation}

    Note that we can re-write \eqref{eq:gen1} as follows

    \begin{equation}\label{eq:gen2}
        (x_j+1)\, \mathcal{F}^{(j)}(x_1,\ldots,x_p)  =x_j\, \mathcal{F}(x_1,\ldots,x_p) + x_j \quad \text{for every } 1\le j\le p.
    \end{equation}

    Note that $x_1,\ldots,x_p\neq-1$ since $|x_1|+\ldots+|x_p| < 1$. So, we can rewrite \eqref{eq:gen2} as follows  

    \begin{equation}\label{eq:gen3}
        \mathcal{F}^{(j)}(x_1,\ldots,x_p) =\frac{x_j}{x_j+1}\, \mathcal{F}(x_1,\ldots,x_p) + \frac{x_j}{x_j+1} \quad \text{for every } 1\le j\le p.
    \end{equation}

    Indeed, by adding the equations involved in \eqref{eq:gen3} we obtain
    
    \begin{equation}\label{eq:gen4}
      \mathcal{F}(x_1,\ldots,x_p) = \displaystyle\frac{ \frac{x_1}{x_1+1} + \ldots + \frac{x_p}{x_p+1} }{ 1 - \frac{x_1}{x_1+1} - \ldots - \frac{x_p}{x_p+1} } = \sum_{n\ge1} \left( \frac{x_1}{x_1+1} + \ldots + \frac{x_p}{x_p+1} \right)^n
    \end{equation}   
\end{proof}

    Note that $\mathcal{F}(x_1,\ldots,x_p)$ converges (absolutely) if and only if $\displaystyle \left| \frac{x_1}{x_1+1} + \ldots + \frac{x_p}{x_p+1} \right|<1$. Indeed, we have that $\mathcal{F}(x_1,\ldots,x_p)$ converges absolutely on
    \[
     \Omega:=\left\{(x_1,\ldots,x_p)\in\mathbb{R}^p\,:\, \displaystyle |x_1|+\ldots+|x_p|<1 \, ,\, \left| \frac{x_1}{x_1+1} + \ldots + \frac{x_p}{x_p+1} \right|<1 \right\}.
    \]
    Note that there is a $p$-dimensional ball centered at the origin included in $\Omega$. 
    Now, using the closed formula of $\mathcal F$, we can obtain a closed formula for $X_{k_1,\ldots,k_p}$, see the following result.

\begin{theorem}\label{th:Dijk}
 For every $k_1,\ldots,k_p\in\mathbb{N}^+$ we have 
  \begin{equation}\label{eq:Dijk}
    X_{k_1,\ldots,k_p}= (-1)^{k_1+\ldots+k_p} \displaystyle\sum_{r_1=1}^{k_1}\ldots \sum_{r_p=1}^{k_p} {r_1+\ldots+r_p \choose r_1,\ldots,r_p} \prod_{1\le i\le p}  (-1)^{r_i} {k_i-1\choose r_i-1}.
 \end{equation}
\end{theorem}

\begin{proof}
    $X_{k_1,\ldots,k_p}$ is the coefficient of $x_1^{k_1}x_2^{k_2}\ldots x_p^{k_p}$ in $\mathcal{F}(x_1,\ldots,x_p)$. We also have
    \[
      \frac{z}{z+1}=\displaystyle\sum_{n\ge1} (-1)^{n-1}z^n \qquad \text{ for every } |z|<1. 
    \]
    Moreover, using the Taylor polynomial of $\frac{z}{z+1}$ with Peano's form of remainder, we have
    \[
      \frac{z}{z+1}= z - z^2 + z^3 - \ldots - (-z)^{k} \, + \, \mathcal{O}(z^{k+1})  \quad \text{for every } k\ge1. 
    \]
    Note that $D(k,r)$ is also the coefficient of $z^k$ in  $\left(z - z^2 + z^3 - \ldots - (-z)^{k} \right)^r$. 
    Thus, from \eqref{eq:gen4} we have
    \[
      \mathcal{F}(x_1,\ldots,x_p) = \sum_{n\ge1} \big( x_1 - \ldots - (-x_1)^{k_1} + \mathcal{O}(x_1^{k_1+1}) \, + \, \ldots \, +\, x_p - \ldots - (-x_p)^{k_p} + \mathcal{O}(x_p^{k_p+1}) \big)^n
    \]
    Thus, we have that $X_{x_1,\ldots,k_p}$ is the coefficient of $x_1^{k_1}x_2^{k_2}\ldots x_p^{k_p}$ in
    \[
     \begin{aligned}
       & \displaystyle\sum_{n=p}^{k_1+\ldots+k_p} \big( x_1 - \ldots - (-x_1)^{k_1} \, + \, x_2 - \ldots - (-x_2)^{k_2} \, + \, \ldots \, + \, x_p - \ldots - (-x_p)^{k_p} \big)^n \\
       = & \displaystyle\sum_{n=p}^{k_1+\ldots+k_p} \sum_{r_1+\ldots+r_p=n} {n \choose r_1,\ldots,r_p} \big(x_1 - \ldots - (-x_1)^{k_1}\big)^{r_1}\big(x_2 - \ldots - (-x_2)^{k_2}\big)^{r_2}\ldots \big(x_p - \ldots - (-x_p)^{k_p}\big)^{r_p}
     \end{aligned}
    \]

    Then, by adding the coefficients of $x_1^{k_1}x_2^{k_2}\ldots x_p^{k_p}$ for each $n$ we obtain 
    \[
     X_{k_1,\ldots,k_p}= \displaystyle\sum_{n=1}^{k_1+\ldots+k_p} \sum_{r_1+\ldots+r_p=n} {n \choose r_1,\ldots,r_p} D(k_1,r_1)\,D(k_2,r_2)\,\ldots\,D(k_p,r_p)
    \]

%
Now, we can use combinatorial arguments to obtain a recurrence relation involving $\{D(k,r)\}_{k\ge r}$ and initial conditions that allow us to solve  $\{D_{k,r}\}_{k\ge r}$. 
Note that we have trivial relations on the double indices sequence that could work as initial conditions 
\begin{equation}\label{eq:rec_initials}
    D(k,0)=0, \, \forall k\in\mathbb{N}, \quad D(k,r)=0, \text{ if } k<r, \quad D(k,1)=(-1)^{k-1}, \, \forall k>0, \quad D(k,k)=1,\, \forall k>0.
\end{equation}
Using the fact that
\[
 \big(z-z^2+z^3-z^4+\ldots-(-z)^{k}\big)^{r+1} = \big(z-z^2+z^3-z^4+\ldots-(-z)^{k}\big)^{r} \, \big(z-z^2+z^3-z^4+\ldots-(-z)^{k}\big),
\]
we can also obtain the following recurrence relation that could be used to obtain $D(k,r)$ for whatever pair $(k,r)$ whenever $k\ge r$
\begin{equation}\label{eq:rec_relations}
    D(k,r+1) = D(k-1,r) - D(k-2,r) + D(k-3,r) - ... (-1)^{r-1} D(k-r,r).
\end{equation}
Hence, we have that \eqref{eq:rec_initials} and \eqref{eq:rec_relations} give an iterative way to solve $\{D_{k,r}\}_{k\ge r}$. 
Moreover, subtracting \eqref{eq:rec_relations} from the identity $D(k,r)=D(k,r)$, we obtain
\begin{equation}\label{eq:rec_D}
    D(k+1,r+1)=D(k,r)-D(k,r+1)
\end{equation}
Let $E(a,b)$ be the sequence that verifies 
\[
D(k,r):= (-1)^{k+r}E(k-1,k-r)
\]
Then, from \eqref{eq:rec_initials} we have the following initial conditions for $E(a,b)$
\begin{equation}\label{eq:rec_initialsE}
    E(a-1,a)=0, \quad
    E(a,-b)=0, \quad E(a,a)=1, \quad E(a,0)=1,\qquad \forall a\in\mathbb{N},\forall b>0.
\end{equation}
Moreover, we obtain the following recurrence relation for $E(a,b)$ from \eqref{eq:rec_D}
\begin{equation}\label{eq:rec_E}
    E(a+1,b)=E(a,b)+E(a,b-1) \quad \text{for every  } a,b\in\mathbb{N}.
\end{equation}
Then, uniqueness of $E(a,b)$ satisfying \eqref{eq:rec_initialsE} and \eqref{eq:rec_E} gives
\[
 E(a,b)={a \choose b} \quad \text{for every  } a,b\in\mathbb{N}.
\]
\end{proof}


Indeed, we have some identities involving $\left\{X_{k_1,\ldots,k_p}\right\}$ due to combinatorial arguments. For instance,
since $X_{k_1,\ldots,k_p}=0$ if $(k_1,\ldots,k_p)\notin\mathbb{T}$, we have, \emph{e.g.},
\[
 X_{k+2,k}=\displaystyle\sum_{n=1}^{2k+2} (-1)^{n} \sum_{r=0}^{n} {n \choose r}\, {k+1\choose k+2-r}\,{k-1 \choose k-n+r}=0 \qquad \text{for every } k\in\mathbb{N}.
\]
and since $X_{k,k}=2$ for $k>0$ and $X_{k,k+1}=1$ for $k\ge0$ we also have
\[
 \displaystyle\sum_{n=1}^{2k+1} (-1)^{n+1} \sum_{r=0}^{n} {n \choose r} \, {k \choose k+1-r} \, {k-1 \choose k-n+r} =1 \quad \text{for every } k\in\mathbb{N}
\]
and 
\[
 \displaystyle\sum_{n=1}^{2k} (-1)^{n} \sum_{r=0}^{n} {n \choose r}\, {k-1 \choose k-r}\, {k-1 \choose k-n+r}  \, =\, 2 \quad \text{for every } k\in\mathbb{N}^+.
\]

The following result shows a closed formula of the number of acyclic orientation of a complete multipartite graphs with labelled vertices, see Theorem \ref{th:acyc_labelled}. 

\begin{theorem}\label{th:acyc_labelled}
The number of acyclic orientation of a complete multipartite graph $K_{n_1,n_2,\ldots,n_p}$ with labelled vertices and $p$ parts with sizes of the parts $n_1,n_2,\ldots,n_p$, respectively, is
 \begin{equation}\label{eq:acyc_labelled}
     B_{n_1,n_2,\ldots,n_p}= \displaystyle\sum_{k_1\le n_1}\sum_{k_2\le n_2}\ldots\sum_{k_p\le n_p} k_1!k_2!\ldots k_p! {n_1 \brace k_1}{n_2\brace k_2}\ldots{n_p\brace k_p} \, X_{k_1,\ldots,k_p}
 \end{equation}
where ${r\brace s}$ denotes the Stirling number of the second kind. 
\end{theorem}

\begin{proof}
 Notice that we can use the same argument as in Proposition \ref{p:labelled}. We may count the number of acyclic orientations of $K_{n_1,n_2,\ldots,n_p}$ by grouping them according to $k_{i-1}$, the number of groups of $i$'s codes (code $i$ associated to part with size $n_{i+1}$) separated by at least another group of code in $[p]$. 
 Then, consider the distinct ways to obtain $k_i$ groups out of $n_i$ codes $i-1$, \emph{i.e.}, ${n_i\brace k_i}$, and its corresponding permutation of the $k_i$ groups, \emph{i.e.}, $k_i!$.
\end{proof}

The length of the longest path in an acyclic orientation of a labelled complete multipartite graph was discussed in \cite{KLM}. In this direction, we have the following result below.
In an acyclic orientation of a complete multipartite graph, the longest directed paths always start from the part that includes the sources, and end at the part that includes the sinks. Note that a longest path cannot start from another part of the multipartite graph since a source makes it one edge longer. Analogously, a sink could make a path one edge longer if it doesn't end in the part that includes the sinks.
Moreover, the codification \eqref{eq:code_function}, given in Theorem \ref{th:coding}, gives a partition on the vertices of $\mathcal K$ induced by the equivalence relation $R_{\mathcal{K}}$ defined by: 
 
 \emph{Two vertices are related by $R_{\mathcal K}$ if they are sources in some subsequent acyclic orientation obtained during the source removing decomposition, {\it i.e.}, if the two vertices are represented in the code within a sub-string of consecutive and identical codes.}

 Notice that since the code assigned to $\mathcal K$ is unique when you fix the code assigned of each part, \emph{i.e.}, unique unless you consider permutation of their parts. Then, we may verify that $R_{\mathcal K}$ is an equivalence relation on the set of vertices of the complete multipartite graph. 
 Indeed, the code of $\mathcal K$ induces a total order $\prec_\mathcal{K}$ in the partition $V(K_{n_1,n_2,\ldots,n_p})/R_{\mathcal K}$ given by the order of appearance on the code of $\mathcal K$. Note that $|V(K_{n_1,n_2,\ldots,n_p})/R_{\mathcal K}|=k_1+k_2+\ldots+k_p$ whenever one of the codifications given by \eqref{eq:code_function} of $\mathcal K$ is counted in $X_{k_1,k_2,\ldots,k_p}$.

\begin{proposition}\label{p:length}
  The length of the longest path in an acyclic orientation $\mathcal K$ of a complete multipartite graph $K_{n_1,n_2,\ldots,n_p}$ is the size of the partition induced by $R_{\mathcal K}$ minus one, {\it i.e.}, 
  $|V(K_{n_1,n_2,\ldots,n_p})/R_{\mathcal K}|-1$.

  Furthermore, the number of longest paths in $\mathcal K$ is given by multiplying the sizes of each part of the partition induced by $R_{\mathcal K}$.
\end{proposition}

\begin{table}[h]
\centering
\setlength{\tabcolsep}{1mm} 
\def\arraystretch{1.0} 
\begin{tabular}{|c|c|}
\hline
\begin{tikzpicture}[thick,
  every node/.style={draw,circle},
  fsnode/.style={fill=myblue},
  ssnode/.style={fill=mygreen},
  every fit/.style={circle,draw,inner sep=2pt,text width=1.5cm},
  ->,shorten >= 3pt,shorten <= 3pt
]
\begin{scope}[start chain=going below]
\foreach \i in {1,2,3}
  \node[fsnode,on chain] (f\i) [label=left: \i] {};
\end{scope}
\begin{scope}[xshift=2.9cm,yshift=-.1cm,start chain=going below,node distance=7mm]
\foreach \i in {4}
  \node[ssnode,on chain] (s\i) [label=right: \i] {};
\end{scope}
\begin{scope}[xshift=2.9cm,yshift=-2.1cm,start chain=going below,node distance=7mm]
\foreach \i in {5}
  \node[ssnode,on chain] (t\i) [label=right: \i] {};
\end{scope}

\node [myblue,fit=(f1) (f3),label=above:$n_0$] {};
\node [mygreen,fit=(s4) ,label=above:$n_1$] {};
\node [mygreen,fit=(t5) ,label=below:$n_2$] {};

\draw (f1) -- (s4);
\draw (f1) -- (t5);
\draw (f2) -- (s4);
\draw (f2) -- (t5);
\draw (f3) -- (s4);
\draw (f3) -- (t5);
\draw (s4) -- (t5);
\end{tikzpicture}
&\begin{tikzpicture}[thick,
  every node/.style={draw,circle},
  fsnode/.style={fill=myblue},
  ssnode/.style={fill=mygreen},
  every fit/.style={circle,draw,inner sep=2pt,text width=1.5cm},
  ->,shorten >= 3pt,shorten <= 3pt
]
\begin{scope}[start chain=going below]
\foreach \i in {1,2,3}
  \node[fsnode,on chain] (f\i) [label=left: \i] {};
\end{scope}

\begin{scope}[xshift=2.9cm,yshift=-.1cm,start chain=going below,node distance=7mm]
\foreach \i in {4}
  \node[ssnode,on chain] (s\i) [label=right: \i] {};
\end{scope}

\begin{scope}[xshift=2.9cm,yshift=-2.1cm,start chain=going below,node distance=7mm]
\foreach \i in {5}
  \node[ssnode,on chain] (t\i) [label=right: \i] {};
\end{scope}

\node [myblue,fit=(f1) (f3),label=above:$n_0$] {};
\node [mygreen,fit=(s4) ,label=above:$n_1$] {};
\node [mygreen,fit=(t5) ,label=below:$n_2$] {};

\draw (f2) -- (s4);
\draw (s4) -- (t5);

\end{tikzpicture}
\\ \( 0.0.0.1.2\) 
\\ \hline
\begin{tikzpicture}[thick,
  every node/.style={draw,circle},
  fsnode/.style={fill=myblue},
  ssnode/.style={fill=mygreen},
  every fit/.style={circle,draw,inner sep=2pt,text width=1.5cm},
  ->,shorten >= 3pt,shorten <= 3pt
]
\begin{scope}[start chain=going below]
\foreach \i in {1,2,3}
  \node[fsnode,on chain] (f\i) [label=left: \i] {};
\end{scope}
\begin{scope}[xshift=2.9cm,yshift=-.1cm,start chain=going below,node distance=7mm]
\foreach \i in {4}
  \node[ssnode,on chain] (s\i) [label=right: \i] {};
\end{scope}
\begin{scope}[xshift=2.9cm,yshift=-2.1cm,start chain=going below,node distance=7mm]
\foreach \i in {5}
  \node[ssnode,on chain] (t\i) [label=right: \i] {};
\end{scope}

\node [myblue,fit=(f1) (f3),label=above:$n_0$] {};
\node [mygreen,fit=(s4) ,label=above:$n_1$] {};
\node [mygreen,fit=(t5) ,label=below:$n_2$] {};

\draw (s4) -- (f1);
\draw (f2) -- (s4);
\draw (f3) -- (t5);
\draw (t5) -- (f2);
\draw (t5) -- (f1);
\draw (f3) -- (s4);
\end{tikzpicture}
&\begin{tikzpicture}[thick,
  every node/.style={draw,circle},
  fsnode/.style={fill=myblue},
  ssnode/.style={fill=mygreen},
  every fit/.style={circle,draw,inner sep=2pt,text width=1.5cm},
  ->,shorten >= 3pt,shorten <= 3pt
]
\begin{scope}[start chain=going below]
\foreach \i in {1,2,3}
  \node[fsnode,on chain] (f\i) [label=left: \i] {};
\end{scope}
\begin{scope}[xshift=2.9cm,yshift=-.1cm,start chain=going below,node distance=7mm]
\foreach \i in {4}
  \node[ssnode,on chain] (s\i) [label=right: \i] {};
\end{scope}
\begin{scope}[xshift=2.9cm,yshift=-2.1cm,start chain=going below,node distance=7mm]
\foreach \i in {5}
  \node[ssnode,on chain] (t\i) [label=right: \i] {};
\end{scope}

\node [myblue,fit=(f1) (f3),label=above:$n_0$] {};
\node [mygreen,fit=(s4) ,label=above:$n_1$] {};
\node [mygreen,fit=(t5) ,label=below:$n_2$] {};

\draw (s4) -- (f1);
\draw (f2) -- (s4);
\draw (f3) -- (t5);
\draw (t5) -- (f2);
\end{tikzpicture}
\\ \( 0.2.0.1.0\) 
\\ \hline
\end{tabular}
\caption{Two acyclic orientations of $K_{3,1,1}$ and their corresponding longest paths.}
\label{tab:paths}
\end{table}

\begin{proof}
 Consider now a path $\mathcal P$ in $\mathcal K$. Notice that $\mathcal P$ contains at most 1 (one) vertex of each part of the partition; otherwise, there is not a directed path joining them. Moreover, a path containing vertices in each part of the given partition $V_{K_{n_1,n_2,\ldots,n_p}}/ R_{\mathcal K}$ contains the maximum number of vertices and therefore, is one of the longest directed paths in $\mathcal K$. Note that such a $\mathcal P$ is one of the longest paths in $\mathcal K$. Therefore, a longest path includes $k_1+\ldots+k_p$ vertices, and consequently, its length is $k_1+\ldots+k_p-1$. 
 

 Indeed, the number of longest paths in $\mathcal K$ is given by the multiplication of the size of each part of the partition induced by $R$ since we may choose a random vertex from each element of $V_{K_{n_1,n_2,\ldots,n_p}}/R$, see Table \ref{tab:paths}.
\end{proof}

Notice that we may use Proposition \ref{p:length} to investigate the distribution of the longest paths in the acyclic orientations of a complete multipartite graphs. 
Proposition \ref{p:length} and Gallai-Hasse-Roy-Vitaver theorem give the trivial result
\[
  \chi(k_{n_1,\ldots,k_p})=\displaystyle\min_{\mathcal{K}} |V(K_{n_1,n_2,\ldots,n_p})/R_{\mathcal K}|=p.
\]

\section{Acyclic orientations of a graph with labelled vertices}\label{sect_general}

In this Section we deal with the acyclic orientations of a general graph.  
First, notice that, such source removing decomposition defined in section above exists for every acyclic orientation of a given graph. Moreover, for each acyclic orientation $\mathcal K$ of a graph $G$, we can similarly define $R_{\mathcal K}$ as 

 \emph{Two vertices are related by $R_{\mathcal K}$ if they are sources in some subsequent acyclic orientation obtained during the source removing decomposition.}

Clearly, $R_{\mathcal K}$ defines a total order in the partition $V(G)/R_{\mathcal K}$ by considering the order of appearance within the source removing decomposition in $G$. Thus, we can define a total order in $V(G)$, by considering the order above as primary rule and a given (fixed) total order of the vertices as a secondary rule. For instance, in Table \ref{tab:paths}, the acyclic orientation coded $0.0.0.1.2$ has primary rule $n_0\prec n_1\prec n_2$ and we may consider as secondary rule the natural order in $V(G)=\{1,2,3,4,5\}$, then it induces the order $1\prec2\prec3\prec4\prec5$ in $V(G)$; however, the acyclic orientation coded $0.2.0.1.0$ is determined by the primary rule (\emph{i.e.}, $\{3\}\prec\{5\}\prec\{2\}\prec\{4\}\prec\{1\}$) which induces the order $3\prec5\prec2\prec4\prec1$ in $V(G)$. 

In other words, given a total order in $V(G)$ and an acyclic orientation of $G$, we may assign orientation to edges in $\overline G$ keeping the condition of the acyclic orientation, even until we obtain an acyclic orientation of a complete graph in a unique way, \emph{i.e.}, assigning orientation to each additional edges according to the given total order in $V(G)$. 

\begin{theorem}\label{thm:perm_code}
  For every acyclic orientation $\mathcal K$ of a graph $G$ with order $n$, we can label the vertices of $G$ as $\{1,2,3,\ldots,n\}$ and define a code $f(\mathcal K)$ as above, such that $f(\mathcal K)$ is a permutation of $(1,2,3,\ldots,n)$.
\end{theorem}

Note that the permutation code defined in Theorem \ref{thm:perm_code} could also be defined via the sink decomposition instead of the source decomposition. However, the sink code is not the reverse of the source code even if we place the sinks in decreasing order, for instance, the directed graph in Figure \ref{fig:SinkCode} has source permutation code $(1,2,3,4,5,6)$, however, its sink permutation code in decreasing order is $(6,5,2,4,3,1)$.

\begin{figure}[ht]
\setlength{\tabcolsep}{1mm} 
\def\arraystretch{1.0} 
\centering

\begin{tikzpicture}[thick,
  every node/.style={draw,circle},
  fsnode/.style={fill=myblue},
  ssnode/.style={fill=mygreen},
  every fit/.style={ellipse,draw,inner sep=-2pt,text width=2cm},
  ->,shorten >= 3pt,shorten <= 3pt
]
\begin{scope}[xshift=-3cm,yshift=0.5cm,start chain=going below,node distance=7mm]
\foreach \i in {1}
  \node[fsnode,on chain] (f\i) [label=left: \i] {};
\end{scope}

\begin{scope}[xshift=0cm,yshift=0.5cm,start chain=going below,node distance=9mm]
\foreach \i in {2,3,4}
  \node[fsnode,on chain] (r\i) [label=above: \i] {};
\end{scope}

\begin{scope}[xshift=3cm,yshift=0.5cm,start chain=going below,node distance=7mm]
\foreach \i in {5,6}
  \node[ssnode,on chain] (s\i) [label=right: \i] {};
\end{scope}

\draw (f1) -- (r2);
\draw (f1) -- (r3);
\draw (f1) -- (r4);
\draw (r3) -- (s5);
\draw (r4) -- (s6);
\end{tikzpicture}

\caption{Acyclic orientation with permutation (sources) code $(1,2,3,4,5)$ and permutation (sinks) code $(6,5,2,4,3,1)$}
\label{fig:SinkCode}
\end{figure}

The following two algorithms allows us to obtain the permutation code of an acyclic orientation of a general graph and the acyclic orientation associated to a permutation code whenever the code is the image of an acyclic orientation of a graph.

\begin{algorithm}\label{alg:coding}
   \emph{Permutation coding}
    \begin{description}
        \item[\it input:] \emph{$\mathcal K$, an acyclic orientation of a graph, \emph{e.g.}, $A_\mathcal{K}$, the adjacency matrix of $\mathcal K$}
        \item[\it output:] \emph{$v$ the permutation code of $\mathcal K$} 
    \end{description}

    \begin{description}
        \item[Step 0:] {Set the output vector $v$ as an empty list.}
        \item[Step 1:] {Identify all sources of $\mathcal K$ analyzing the vertices in order, and insert them into $v$ in the order they appear. If no sources is found then $\mathcal K$ is not an acyclic orientation, STOP.}
        \item[Step 2:] {Remove from $\mathcal K$ all identified sources in Step 1.} 
        \item[Step 3:] {If $K$ is the null graph, STOP, else go to Step 1.} 
    \end{description}
\end{algorithm}

The algorithm above is $\mathcal{O}(n^2)$, in a standard version using the adjacency matrix of the directed graph. However, the number of operations could be reduced if we consider the structure as follows, a new auxiliary vertex as the root of the tree linked to the sources of $\mathcal K$. Then, we obtain a directed tree, and listing the vertices of the tree by levels (without the root) we obtain the permutation code of $\mathcal K$.       

\begin{algorithm}\label{alg:decoding}
   \emph{Decoding a permutation code}
    \begin{description}
        \item[\it input:] \emph{a permutation code $v$, and a graph $G$, \emph{e.g.}, $A_G$ the adjacency matrix of $G$}
        \item[\it output:] \emph{the acyclic orientation $\mathcal K$ such that $f(\mathcal K)=v$ if it exits} 
    \end{description}

    \begin{description}
        \item[Step 0:] {Set the auxiliary lists \emph{prev} and \emph{current} as empty.}
        \item[Step 1:] {Insert in \emph{current} the consecutive vertices in $v$ that appear in increasing order and form an independent set of vertices.}
        \item[Step 2:] {Modify $A_G$ by providing orientation to all edges adjacent to \emph{current} by converting all vertices of \emph{current} in sources.}
        \item[Step 3:] {Remove the vertices in \emph{current} from $v$, \emph{prev}:=\emph{current} and empty \emph{current}.}
        \item[Step 4:] {If $v$ is empty, then STOP, else insert in \emph{current} the consecutive vertices in $v$ that appear in increasing order and form an independent set of vertices and has at least an adjacent vertex in \emph{prev}.} 
        \item[Step 5:] {If \emph{current} is empty, then $v$ is not a code of an acyclic orientation of $G$, STOP; else go to Step 2} 
    \end{description}
\end{algorithm}

For instance, for the two acyclic orientations $0.0.0.1.2$ and $0.2.0.1.0$ in Table \ref{tab:paths}, we have $f(0.0.0.1.2)=(1,2,3,4,5)$ and $f(0.2.0.1.0)=(3,5,2,4,1)$. The function $f$ is one-to-one. Note that the permutations of $(1,2,3,\ldots,n)$ are in a one-to-one correspondence with the acyclic orientations of the complete graph $K_n$ by orienting the edges according to the order induced by the permutation. Besides, from each acyclic orientation of $K_n$, a permutation code, we can remove the edges joining each pair of non-adjacent vertices in $G$ and obtain a unique acyclic orientation of $G$. However, not every permutation code is an image of an acyclic orientation of $G$ mapped by $f$ if $G\not\simeq K_n$. For instance, in the cycle $C_4$ with $V(C_4)=\{1,2,3,4\}$ and $E(C_4)=\{12,23,34,41\}$ the permutation code $(1,4,2,3)$ is not an image of $f$, although by removing its edges in $\overline{C_4}$ we obtain an acyclic orientation of $C_4$. Note that the acyclic orientation obtained in this case has image $(1,2,4,3)$ via $f$ since $2$ and $4$ are in the same class of equivalence induced by that obtained acyclic orientation of $C_4$. Similarly, we may use this idea of the permutation code of an acyclic orientation of $G$ and removing edges from an acyclic orientation of a supergraph of $G$ to obtain the following result.

\begin{proposition}\label{prop:subgraphs}
  Let $G$ and $H$ be two graphs such that $G$ is a subgraph of $H$. Then, the number of acyclic orientations of $G$ is less than or equal to the number of acyclic orientations of $H$. 
  
  Furthermore, if $G$ is a proper subgraph of $H$, then the inequality is strictly less than.
\end{proposition}

For every graph $G$ with labelled vertices, we define the number of acyclic orientations of $G$, by $A(G)$. Hence, we have the following result that gives sharp bounds for $A(G)$.

\begin{theorem}\label{th:wbounds}
    Let $G$ be a connected graph of order $n$. Then, we have
    \[
     2^{n-1}\leq A(G)\leq n!
    \]
\end{theorem}

\begin{proof}
    On one hand, it is clear that $G$ is a subgraph of $K_n$. Thus, by Proposition \ref{prop:subgraphs} we have that $A(G)\leq A(K_n)=n!$. On the other hand, consider $T$ a spanning tree of $G$. So, $T$ is a subgraph of $G$ that has no cycles, therefore, by assigning orientation to each edge of $T$ we always obtain an acyclic orientation of $T$ since trees have no cycles. Indeed, $A(T)=2^{n-1}$ since the size of $T$ is $n-1$ and each edge has two options of orientating. Then, by Proposition \ref{prop:subgraphs} we have $2^{n-1}=A(T)\leq A(G)$. 
\end{proof}

It is clear that a graph with a proper coloring and with the minimum number of colors such that the monochromatic parts have sizes $n_1,n_2,\ldots,n_p$, respectively, is a subgraph of the complete multipartite graph
 $K_{n_1,n_2,\ldots,n_p}$. Therefore, Proposition \ref{prop:subgraphs} has the following consequence. 

\begin{proposition}\label{l:UpperB}
    Let $G$ be a graph with chromatic number $p$. Consider there is a proper coloring of $G$ with the minimum number of colors such that the monochromatic parts of the coloring have sizes $n_1,n_2,\ldots,n_p$, respectively. Then, we have

    \begin{equation}
        A(G)\leq B_{n_1,n_2,\ldots,n_p}.
    \end{equation}

    Furthermore, the equality is attained if and only if $G\simeq K_{n_1,n_2,\ldots,n_p}$. 
\end{proposition}

Notice that for some graphs there are many proper colorings with the minimum number of colors, so the result above could be written as a minimum over several $B_{n_1,n_2,\ldots,n_p}$.  
In addition, the Proposition \ref{l:UpperB} has the following consequence.

\begin{proposition}\label{prop:Bounds}
    Let $G$ be a graph with chromatic number $p$ and independence number $r$. Then, we have
    \begin{equation}\label{eq:101}
        A(G)\leq B_{p\times r}
    \end{equation}
   Furthermore, the equality is attained if and only if $G$ is isomorphic to the Tur\'an graph of order $pr$ and $p$ parts. 
\end{proposition}

\begin{proof}
    Consider a proper coloring in $G$ with the minimum number of colors and size of the parts $n_1,n_2,\ldots,n_p$, respectively. Note that in every proper coloring of a graph all monochromatic parts are an independent set of vertices; therefore, each of those parts has size less than or equal to the independence number of $G$. Indeed, we have $n_i\leq r$ for every $1\leq i\leq p$. By \eqref{eq:acyc_labelled} we have that $B_{n_1,n_2,\ldots,n_p}\leq B_{m_1,m_2,\ldots,m_p}$ if $n_i\leq m_i$ for every $1\leq i\leq p$. Therefore, by Proposition \ref{l:UpperB} we have
    \[
     A(G)\leq B_{n_1,n_2,\ldots,n_p}\leq B_{p\times r}.
    \]
    Clearly, we have that $K_{\underbrace{r,r,\ldots,r}_{p \text{ times}}}$ has chromatic number $p$, independence number $r$ and provide the equality of \eqref{eq:101}; on the other hand, if $G$ is not isomorphic to such a Tur\'an graph, then $G$ is a proper subgraph of the Tur\'an graph and Proposition \ref{prop:subgraphs} gives a strict inequality. 
\end{proof}

\section*{Acknowledgements}

The first author was supported by a grant from Agencia Estatal de Investigaci\'on (PID2019-106433GB-I00 /AEI/10.13039/501100011033), Spain.

\end{document}